\documentclass[a4,12pt]{amsart}
\usepackage{amssymb,amstext,amsmath,amscd,amsthm,
amsfonts,enumerate,graphicx,latexsym}
\usepackage{url}
\usepackage[usenames]{color}
\definecolor{gr}{rgb}{0.1, .5 , .10}
\usepackage[all]{xy}

\newtheorem{theorem}{Theorem}[section]
\newtheorem{theorem*}{Theorem}
\newtheorem{corollary}[theorem]{Corollary}
\newtheorem{corollary*}[theorem*]{Corollary}
\newtheorem{lemma}[theorem]{Lemma}
\newtheorem{proposition}[theorem]{Proposition}

\newtheorem{maintheorem}{Theorem}
\theoremstyle{definition}

\newtheorem{remark}[theorem]{Remark}

\newtheorem*{question*}{Question}

\newtheorem*{conjecture*}{Conjecture}
\newtheorem{example}[theorem]{Example}

\newtheorem*{notation*}{Notation}

\newtheorem*{claim*}{Claim}

\numberwithin{equation}{theorem}
\def\rad{\operatorname{rad}}

\def\Hom{\operatorname{Hom}}

\def\End{\operatorname{End}}

\def\Coker{\operatorname{Coker}}
\def\identity{\mathrm{id}}
\def\thick{\operatorname{\mathsf{thick}}}
\def\add{\operatorname{\mathsf{add}}}

\def\mod{\operatorname{\mathsf{mod}}}

\def\proj{\operatorname{\mathsf{proj}}}

\def\Kb{\mathsf{K^b}}

\def\sttilt{\operatorname{\mathsf{s\tau -tilt}}}
\def\silt{\operatorname{\mathsf{silt}}}

\def\C{\mathcal{C}}

\def\D{\mathcal{D}}
\def\T{\mathcal{T}}

\def\Y{\mathcal{Y}}

\newcommand{\spq}{{\rm sp}}
\newcommand{\old}[1]{{\color{red} #1}}

\begin{document}
%\allowdisplaybreaks
\setlength{\baselineskip}{15pt}
\title{$\tau$-tilting finite triangular matrix algebras}
\author{Takuma Aihara}
\address{Department of Mathematics, Tokyo Gakugei University, 4-1-1 Nukuikita-machi, Koganei, Tokyo 184-8501, Japan}
\email{aihara@u-gakugei.ac.jp}
\author{Takahiro Honma}
\address{Graduate School of Mathematics, Tokyo University of science, 1-3 Kagurazaka, Shinjuku, Tokyo 162-8601, Japan}
\email{1119704@ed.tus.ac.jp}

\thanks{TA was partly supported by JSPS Grant-in-Aid for Young Scientists 19K14497.}
\begin{abstract}
First, we give a new example of silting-discrete algebras.
Second, one explores when the algebra of triangular matrices over a finite dimensional algebra is $\tau$-tilting finite.
In particular, we classify algebras over which triangular matrix algebras are $\tau$-tilting finite.
Finally, we investigate when a triangular matrix algebra is silting-discrete. 
\end{abstract}
\maketitle
%\tableofcontents
%%%%%%%%%%%%%%%%%%%%%%%%%%%%%%%%%%%%%%%%%%%%%%%%%%%%%%%%
\section{Introduction}

Silting objects play a central role in tilting theory to describe the structure of derived categories and control derived equivalences.
One of the most crucial purposes is to clarify the whole picture of silting objects.
To realize the goal, we first discuss when a triangulated category is silting-discrete; roughly speaking, the silting-discreteness is the finiteness of silting objects.
A finite dimensional algebra is said to be \emph{silting-discrete} if the perfect derived category is silting-discrete; for example, the following algebras are silting-discrete:
\begin{itemize}
\item representation-finite piecewise hereditary algebras \cite{Ai};
\item representation-finite symmetric algebras \cite{Ai};
\item Brauer graph algebras whose Brauer graphs have at most one odd cycle and none of even cycles \cite{AAC};
\item derived-discrete algebras with finite global dimension \cite{BPP};
\item weakly-symmetric preprojective algebras of Dynkin type \cite{Ai2, AM, AdK, AD};
\item algebras of dihedral, semidihedral and quaternion type \cite{EJR};
\item symmetric algebras of tubular type with nonsingular Cartan matrix \cite{AHMW}.
\end{itemize}

Let $\Lambda$ be a finite dimensional algebra over an algebraically closed field $K$ which is basic and ring-indecomposable.

The first aim of this paper is to construct a new silting-discrete algebra from a given one.
We denote by $\silt\Lambda$ the set of isomorphism classes of basic silting objects of the perfect derived category for $\Lambda$.
Here is the first main theorem.

\begin{maintheorem}[Theorem \ref{sd}]
Let $R$ be a finite dimensional local $K$-algebra and put $\Gamma:=R\otimes_K\Lambda$.
If $\Lambda$ is silting-discrete, then we have a poset isomorphism $\silt\Lambda\to\silt\Gamma$.
In particular, $\Gamma$ is also silting-discrete.
\end{maintheorem}

As an example, we consider the $n\times n$ (upper) triangular matrix algebra $T_n(R)$ over a local algebra $R$, which is actually isomorphic to $R\otimes_KK\overrightarrow{A_n}$.
So, we get a corollary of this theorem (Proposition \ref{polynomial}(1)).

In the context of triangular matrix algebras $T_n(\Lambda)$ over an algebra $\Lambda$ (not necessarily local), it seems to be difficult to understand when $T_n(\Lambda)$ is silting-discrete.
Thus, let us turn our attention to ``two-term'' silting objects.
We say that an algebra is \emph{$\tau$-tilting finite} if there are only finitely many two-term silting objects; see \cite{AIR}.
It is evident that a silting-discrete algebra is $\tau$-tilting finite.
We also know that a representation-finite algebra is $\tau$-tilting finite.
As an analogue of Auslander--Reiten's results in \cite{AR},
we have the second main theorem of this paper.

\begin{maintheorem}[Theorem \ref{2taustfrf}]\label{mt2}
Assume that $\Lambda$ is representation-finite. 
Then we have:
\begin{enumerate}
\item If the Auslander algebra of $\Lambda$ is $\tau$-tilting finite, then so is $T_2(\Lambda)$.
\item If $\Lambda$ is simply-connected, then the converse of (1) holds.
\end{enumerate}
\end{maintheorem}

We give sufficient conditions for algebras to be $\tau$-tilting infinite. In particular, the converse of Theorem \ref{mt2}(1) is not necessarily true. 

\begin{maintheorem}[Theorem \ref{necessary}, Proposition \ref{third},  and Example \ref{ce}]
The following hold:
\begin{enumerate}
\item Assume that the Gabriel quiver of $\Lambda$ has no loop.
If the separated quiver of $\Lambda$ has a connected component which is not of type $A_n$, then $T_2(\Lambda)$ is $\tau$-tilting infinite.
\item The third triangular matrix algebra ${T_2}^3(\Lambda)$ is $\tau$-tilting infinite.
\item There is a representation-finite algebra $\Gamma$ such that $T_2(\Gamma)$ is $\tau$-tilting finite and the Auslander algebra of $\Gamma$ is $\tau$-tilting infinite. 
\end{enumerate}
\end{maintheorem}

We classify algebras $\Lambda$ with $T_n(\Lambda)$ $\tau$-tilting finite.
Here is the fourth main theorem.

\begin{maintheorem}[Theorem \ref{s3} and \ref{s2}]
Let $\Lambda$ be a finite dimensional nonlocal algebra over an algebraically closed field whose Gabriel quiver has no loop and $n\geq3$.
Then the following are equivalent:
\begin{enumerate}
\item $T_n(\Lambda)$ is $\tau$-tilting finite;
\item One of the following cases holds:
\begin{enumerate}
%\item $\Lambda$ is local;
\item $n=4$ and $\Lambda$ is the path algebra of type $A_2$;
\item $n=3$ and $\Lambda$ is a Nakayama algebra with precisely 2 simple modules;
\item $n=3$ and $\Lambda$ is a Nakayama algebra with radical square zero.
\end{enumerate}
\end{enumerate}
\end{maintheorem}

This theorem tells us the fact that for a simply-connected algebra $\Lambda$ and $n\geq3$, $T_n(\Lambda)$ is $\tau$-tilting finite if and only if it is representation-finite (Corollary \ref{sc3}).

Finally, let us go back to the study on the silting-discreteness of $T_n(\Lambda)$.
Although it might be very hard to classify algebras $\Lambda$ with $T_n(\Lambda)$ silting-discrete in general, we try it for a radical-square-zero linear Nakayama algebra $\Lambda$.
Here is the last main theorem.

\begin{maintheorem}[Theorem \ref{lsd}]
Let $\Lambda$ be a radical-square-zero linear Nakayama algebra with $r$ simple modules.
Then $T_n(\Lambda)$ is silting-discrete if and only if one of the following cases occurs:
(i) $n=1$;
(ii) $r=1$;
(iii) $n=2$ and $1<r\leq4$; 
(iv) $1<n\leq4$ and $r=2$.
\end{maintheorem}

Throughout this paper, algebras are always assumed to be finite dimensional over an algebraically closed field $K$.
Modules are finite dimensional and right modules.
For an algebra $\Lambda$, we denote by $\mod\Lambda\ (\proj\Lambda)$ the category of (projective) modules over $\Lambda$.
The perfect derived category of $\Lambda$ is denoted by $\Kb(\proj\Lambda)$.

\section{A new example of silting-discrete algebras}

In this section, we give a new construction of silting-discrete algebras.
Let $\Lambda$ be a ring-indecomposable basic algebra.

Let $\T$ be a triangulated category which is Krull--Schmidt, $K$-linear and Hom-finite.  
We say that an object $T$ is \emph{silting} if it satisfies $\Hom_\T(T, T[i])=0$ for any $i>0$ and $\T=\thick T$.
Here, $\thick T$ stands for the smallest thick subcategory of $\T$ containing $T$.
It is known that the set $\silt\T$ of isomorphism classes of basic silting objects of $\T$ has a partial order $\geq$ and actions $\mu^\pm$ of \emph{silting mutation}; see \cite{AI} for details.
%
\if0
For objects $T$ and $U$, we write $T\geq U$ if $\Hom_\T(T, U[i])=0$ for every $i>0$.
It is known that the relation $\geq$ gives a partial order on  the set $\silt\T$ of isomorphism classes of basic silting objects of $\T$ \cite[Theorem 2.11]{AI}.

We now assume that $\T$ is $K$-linear and Hom-finite.
Then, we can apply silting mutation to any silting object \cite[Theorem 2.31]{AI}.
Let $T$ be a (basic) silting object of $\T$ with decomposition $T=X\oplus M$ (not necessarily $X$ is indecomposable).
First, take a (minimal) left $\add M$-approximation $f:X\to M'$ of $X$; it is always possible because $\T$ is Hom-finite.
Extending $f$ to the triangle $X\xrightarrow{f}M'\to Y\to X[1]$,
we get a new object $\mu_X^-(T):=Y\oplus M$, and call it the \emph{left mutation} of $T$ at $X$.
Dually, define the \emph{right mutation} $\mu_X^+(T)$ of $T$ at $X$.
Then, left and right mutations are again silting.

The \emph{silting quiver} of $\T$ is constructed as follows:
The set of vertices is $\silt\T$, and draw an arrow $T\to U$ whenever $U$ is a left mutation of $T$ at its indecomposable summand, or equivalently $T$ is a right mutation of $U$ at its indecomposable summand.
By \cite[Theorem 2.35]{AI}, the silting quiver is nothing but the Hasse quiver of the poset $\silt\T$.
An interesting question is when the silting quiver is connected;
we say that $\T$ is \emph{silting-connected} if its silting quiver is connected.
\fi
%

A triangulated category $\T$ is said to be \emph{silting-discrete} if it admits a silting object $T$, and for any $n>0$ there are only finitely many (basic) silting objects $U$ satisfying $T\geq U\geq T[n]$.
We obtain from \cite[Corollary 3.9]{Ai} that if $\T$ is silting-discrete, then the Hasse quiver of the poset $\silt\T$ is connected; namely, it is \emph{silting-connected}.

When $\T=\Kb(\proj\Lambda)$ for an algebra $\Lambda$, we write $\silt\T$ by $\silt\Lambda$ and say that $\Lambda$ is \emph{silting-discrete} if $\T$ is silting-discrete.
Here is a new example of silting-discrete algebras.

\begin{theorem}\label{sd}
%Let $\Lambda$ be a silting-discrete algebra and put $\Gamma:=K[x]/(x^2)\otimes_K\Lambda$.
Let $R$ be a local algebra and $\Lambda$ a silting-discrete algebra.
Put $\Gamma:=R\otimes_K\Lambda$.
Then we have a poset isomorphism $\silt\Lambda\to\silt\Gamma$.
In particular, $\Gamma$ is also silting-discrete.
\end{theorem}
\begin{proof}
Let us consider the triangle functor $-\otimes_KR : \Kb(\proj\Lambda)\to\Kb(\proj\Gamma)$.
Then we have an isomorphism $\Hom_{\Kb(\proj\Gamma)}(X\otimes_KR, Y\otimes_KR)\simeq\Hom_{\Kb(\proj\Lambda)}(X,Y)\otimes_KR$.
This leads to the fact that $-\otimes_KR$ keeps the indecomposability of objects; in fact, if $E$ is a local algebra, then so is $E\otimes_KA$ since $K$ is algebraically closed.
We also observe that $-\otimes_KR$ induces an injection $\silt\Lambda\to\silt\Gamma$ preserving the partial order.

We show that $-\otimes_KR$ preserves approximations.
Let $\Y$ be a full subcategory of $\Kb(\proj\Lambda)$ and $f:X\to Y$ be a left $\Y$-approximation of $X$ in $\Kb(\proj\Lambda)$.
For an object $Z$ of $\Y$, the above isomorphism makes a commutative diagram:
\[\xymatrix@C=2cm{
\Hom_{\Kb(\proj\Gamma)}(Y\otimes_KR, Z\otimes_KR) \ar[r]^{-\circ (f\otimes R)} \ar[d]_\simeq & \Hom_{\Kb(\proj\Gamma)}(X\otimes_KR, Z\otimes_KR) \ar[d]^\simeq \\
\Hom_{\Kb(\proj\Lambda)}(Y, Z)\otimes_KR \ar[r]_{(-\circ f)\otimes R} & \Hom_{\Kb(\proj\Lambda)}(X, Z)\otimes_KR
}\]
Since $-\circ f$ is surjective and $-\otimes R$ is (right) exact, we see that the two horizontal arrows are surjections,
whence $f\otimes R$ is a left $\Y\otimes_KR$ approximation of $X\otimes_KR$.

Thus, it turns out that any arrow in $\silt\Lambda$ is also an arrow in $\silt\Gamma$ under the injection $-\otimes_KR:\silt\Lambda\to\silt\Gamma$.
Conversely, we obtain that all paths from/to $\Gamma$ in $\silt\Gamma$ come from those from/to $\Lambda$ in $\silt\Lambda$, because $\Kb(\proj\Lambda)$ and $\Kb(\proj\Gamma)$ have the same rank of the Grothendieck group.

Assume that there is a silting object $U$ of $\Kb(\proj\Gamma)$ with $\Gamma\geq U$ which is out of the image of the functor $-\otimes_KR$.
By \cite[Proposition 2.36]{AI}, we have a path $\Gamma=:U_0\to U_1\to\cdots\to U_\ell \to\cdots$ in $\silt\Gamma$ with $U_i\geq U$ for any $i$, which admits an infinite length, contrary to the assumption of $\Lambda$ being silting-discrete.
Therefore, all silting objects of $\Kb(\proj\Gamma)$ smaller than $\Gamma$ come from those of $\Lambda$.
Then, we derive from \cite[Theorem 2.4]{AM} that $\Gamma$ is silting-discrete.

Finally, we see that the following are equivalent for any $T, U\in\silt\Lambda$:
\begin{enumerate}[(i)]
\item $T\geq U$;
\item there exists a path of finite length from $T$ to $U$;
\item there is a path of finite length from $T\otimes_KR$ to $U\otimes_KR$;
\item $T\otimes_KR\geq U\otimes_KR$.
\end{enumerate}
This implies that the map $-\otimes_KR$ is a poset isomorphism.
%This completes the proof. 
\end{proof}

The \emph{trivial extension} of an algebra $\Lambda$ by a $(\Lambda, \Lambda)$-bimodule $M$ is defined to be $\Lambda\oplus M$ as a $(\Lambda,\Lambda)$-bimodule in which the composition of elements $(a,m)$ and $(b,n)$ is given by $(a,m)\cdot(b,n):=(ab, an+mb)$.
Since the trivial extension of $\Lambda$ by itself is isomorphic to $K[x]/(x^2)\otimes_K\Lambda$, we immediately obtain the following corollary from Theorem \ref{sd}.

\begin{corollary}\label{te}
The trivial extension of $\Lambda$ by itself is silting-discrete if $\Lambda$ is so.
\end{corollary}

\begin{remark}
The trivial extension of $\Lambda$ by its $K$-dual is often called the \emph{trivial extension} of $\Lambda$.
Applying it frequently destroys the silting-discreteness of algebras.
For instance, the algebra given by the quiver $\xymatrix{\bullet \ar@<2pt>[r] & \bullet \ar@<2pt>[l]}$ with radical square zero is silting-discrete, but its trivial extension is neither silting-discrete nor even silting-connected \cite{AGI}.
\end{remark}

We give a slight generalization of \cite[Theorem 15]{EJR}.

\begin{corollary}
Let $p$ be the characteristic of $K$.
Then every $p$-group is contained in the defect group of a nonlocal silting-discrete block of a group algebra.
%Then we can construct a silting-discrete block of a group algebra whose defect group contains arbitrary $p$-group.
\end{corollary}
\begin{proof}
Let $P$ be a $p$-group and $G$ a finite group.
Let $\Lambda$ be a block of the group algebra $KG$ with defect group $D$.
As is well-known, $KP$ is a local algebra and $KP\otimes_K\Lambda$ is a block of the group algebra $K[P\times G]$ whose defect group is $P\times D$.
Thus, we apply Theorem \ref{sd} to get the desired block; for example, if $D$ is cyclic, dihedral, semidihedral or quaternion, then $\Lambda$ is silting-discrete, whence so is $KP\otimes_K\Lambda$.
\end{proof}

\section{Second triangular matrix algebras}

The first aim of this section is to develop the Auslander--Reiten's results in \cite{AR} to the $\tau$-tilting finiteness.
We start with recalling important facts on $\tau$-tilting finite algebras.

Let $\Lambda$ be a ring-indecomposable basic algebra.
%We call a module $M$ \emph{$\tau$-rigid} provided $\Hom_\Lambda(M, \tau M)=0$.
%It is said to be \emph{$\tau$-tilting} if in addition it has the same number of nonisomorhpic indecomposable summands as $\Lambda$.
%A \emph{support $\tau$-tilting} module is defined to be $\tau$-rigid and there exists an idempotent $e$ such that it is $\tau$-tilting over $\Lambda/(e)$.
We call a module $M$ over $\Lambda$ a \emph{support $\tau$-tilting} module provided it is the 0th cohomology of a silting object $T$ in $\Kb(\proj\Lambda)$ with $T^i=0$ unless $i=0,-1$ (see \cite{AIR} for more details).
Our interest in this paper is when an algebra $\Lambda$ has only finitely many support $\tau$-tilting modules; so-called, $\Lambda$ is \emph{$\tau$-tilting finite}.
Evidently, if $\Lambda$ is silting-discrete, then it is $\tau$-tilting finite.
We also know that any factor algebra of a $\tau$-tilting finite algebra is also $\tau$-tilting finite \cite[Theorem 5.12(d)]{DIRRT}.
A module $M$ is said to be \emph{brick} if $\End_\Lambda(M)$ is isomorphic to $K$.
It was shown that $\Lambda$ is $\tau$-tilting finite iff there are only finitely many bricks of $\Lambda$ \cite[Theorem 4.2]{DIJ}.

A main algebra we study here is the $n\times n$ upper triangular matrix algebra $T_n(\Lambda)$,
%$\begin{pmatrix}
%\Lambda & \Lambda & \cdots & \Lambda \\ 0 & \Lambda %\end{pmatrix}$,
which is isomorphic to $\Lambda\otimes_KK\overrightarrow{A_n}$.
Here, $\overrightarrow{A_n}$ denotes the linearly oriented $A_n$-quiver $\xymatrix{1\ar[r] & 2\ar[r] & \cdots \ar[r] & n}$. 
As is well-known, we can identify the category $\mod T_2(\Lambda)$ with the category of homomorphisms in $\mod\Lambda$; that is, the objects are triples $(M,N,f)$ of $\Lambda$-modules $M,N$ and a $\Lambda$-homomorphism $f:M\to N$.
A morphism $(M_1, N_1,f_1)\to(M_2,N_2,f_2)$ is a pair $(\alpha,\beta)$ of $\Lambda$-homomorphisms $\alpha:M_1\to M_2$ and $\beta:N_1\to N_2$ satisfying $f_2\circ\alpha=\beta\circ f_1$.

For an additive category $\C$, we denote by $\mod\C$ the full subcategory of the functor category of $\C$ consisting of finitely generated functors.

Inspired by \cite[Theorem 1.1]{AR}, we have the second main result of this paper.

\begin{theorem}\label{2taustfrf}
Assume that $\Lambda$ is representation-finite.
Then the following hold:
\begin{enumerate}
\item If the Auslander algebra of $\Lambda$ is $\tau$-tilting finite, then so is $T_2(\Lambda)$.
\item If $\Lambda$ is simply-connected, then $T_2(\Lambda)$ is $\tau$-tilting finite if and only if it is representation-finite.
In particular, the converse of (1) holds.
\end{enumerate}
\end{theorem}
\begin{proof}
Let us first recall an argument in \cite[Theorem 1.1]{AR}.
It was shown that the functor $\Phi: \mod T_2(\Lambda) \to \mod(\mod\Lambda)$ sending $(M, N, f)$ to $\Coker\Hom_\Lambda(-,f)$ is full and dense.
Denote by $\D$ the full subcategory of $\mod T_2(\Lambda)$ consisting of modules without indecomposable summands of the forms $(M, M, \identity)$ and $(M, 0,0)$, where $M$ is an indecomposable module over $\Lambda$.
Then the restriction of $\Phi$ is full and dense (not faithful!), and a morphism $\sigma$ in $\D$ with $\Phi(\sigma)$ isomorphic is an isomorphism.

We show the assertion (1) holds true.
As above, any brick over $T_2(\Lambda)$ lying in $\D$ is sent to some brick in $\mod(\mod\Lambda)$ by the functor $\Phi$ and the correspondence is objectively injective.
Therefore, $T_2(\Lambda)$ inherits the finiteness of bricks from the Auslander algebra of $\Lambda$,
whence the assertion follows from \cite[Theorem 4.2]{DIJ}. 

To prove the assertion (2), we assume that $\Lambda$ is simply-connected and $T_2(\Lambda)$ is $\tau$-tilting finite.
Then, $T_2(\Lambda)$ does not contain a finite convex subcategory which is concealed of extended Dynkin type.
The simple-connectedness of $\Lambda$ (i.e. $\widetilde{\Lambda}=\Lambda$ in the sense of \cite{LS1}) implies that $T_2(\Lambda)$ is representation-finite by \cite[Theorem 4]{LS1}.
Moreover, we deduce from \cite[Theorem 1.1]{AR} that the Auslander algebra of $\Lambda$ is also representation-finite, and so it is $\tau$-tilting finite.
\end{proof}

Let $\Lambda$ be an algebra whose Gabriel quiver is $Q$.
The \emph{separated quiver} $Q^\spq$ of $\Lambda$ is defined as follows:
The set of vertices consists of the vertices $i_1,\cdots, i_n$ of $Q$ and their copies $i'_1,\cdots, i'_n$;
we say that $i$ and $i'$ are the same character.
We draw an arrow $i\to k$ if $i$ is a vertex of $Q$, $k=j'$ for some vertex $j$ of $Q$ and if there is an arrow $i\to j$ in $Q$ (see \cite{ARS}).
Observing the separated quiver, we can infer the representation-finiteness and $\tau$-tilting finiteness of a radical-square-zero algebra \cite{ARS, Ad1}; we use their results freely.

We know from \cite[Proposition 3.1]{AR} that if the separated quiver of an algebra $\Lambda$ has a connected component which is not of type $A_n$, then $T_2(\Lambda)$ is representation-infinite.
%However, we can not replace the conclusion with the condition of $T_2(\Lambda)$ being $\tau$-tilting finite.
%Indeed, consider the algebra $\Lambda$ presented by the quiver
%\[\xymatrix{&1\ar@<-2pt>[dl]\ar@<-2pt>[dr]&\\2 \ar@<-2pt>[rr]\ar@<-2pt>[ur]&&3 \ar@<-2pt>[ul]\ar@<-2pt>[ll]}\]
%Then the separated quiver is of type $\widetilde{A_5}$, but $T_2(\Lambda)$ is $\tau$-tilting finite.
Here is a modification to $\tau$-tilting finiteness.

\begin{theorem}\label{necessary}
Let $\Lambda$ be an algebra given by a quiver without loop.
If the separated quiver of $\Lambda$ has a connected component which is not of type $A_n$, then $T_2(\Lambda)$ is $\tau$-tilting infinite.
\end{theorem}
\begin{proof}
Let $\C$ be a connected component of the separated quiver of $\Lambda$.

Suppose that $\C$ is of type $\widetilde{A_n}$.
Then we observe that the (Gabriel) quiver $Q$ of $\Lambda$ admits a subquiver of type $\widetilde{A_n}$, which leads to the fact that the separated quiver of $T_2(\Lambda)$ contains a subquiver of type $\widetilde{D_5}$ with distinct characters.
Hence, it turns out that $T_2(\Lambda)/\rad T_2(\Lambda)$, and so $T_2(\Lambda)$, are not $\tau$-tilting finite.

Thus, we can assume that $\C$ is neither of type $A_n$ nor of type $\widetilde{A_n}$.
This means that $\C$ has a vertex $v$ of degree at least 3;
it does not matter if $v$ is the original or the copy of a vertex of $Q$.
We may suppose that $Q$ possesses no multiple arrow.
As $Q$ admits no loop, it is seen that the 4 points around $v$ (including also $v$) are distinct characters in $\C$.
Therefore, we obtain that the separated quiver of $T_2(\Lambda)$ contains the diagram of type $\widetilde{E_6}$ whose vertices are distinct characters,
whence $T_2(\Lambda)$ is not $\tau$-tilting finite.
\end{proof}

The converse of Theorem \ref{necessary} does not necessarily hold.

\begin{example}[See also Theorem \ref{2taustfrf}(2)]\label{CL}
Let $\Lambda:=K\overrightarrow{A_n}$.
Observe that $T_2(\Lambda)$ is the commutative ladder of degree $n$; see \cite{AHMW, EH, LS1}.
Then the following are equivalent:
(i) $n\leq4$;
(ii) $T_2(\Lambda)$ is representation-finite;
(iii) it is $\tau$-tilting finite.

Combining this observation and Theorem \ref{2taustfrf}(1), we recover \cite[Corollary 4.8]{IX}; that is, the following are equivalent:
(i) $n\leq4$;
(ii) the Auslander algebra of $\Lambda$ is representation-finite;
(iii) it is $\tau$-tilting finite.
\end{example}

We give an example which says that the converse of Theorem \ref{2taustfrf}(1) does not necessarily hold and that the assumption of $\Lambda$ having no loop as in Theorem \ref{necessary} is required.

\begin{example}\label{ce}
Let $\Lambda$ be the radical-square-zero algebra presented by the quiver:
\[\xymatrix{
2 & 1 \ar@(lu,ru)^{}\ar[r]\ar[l] & 3 
}\]
\begin{enumerate}[(i)]
\item The separated quiver of $\Lambda$ consists of three connected components; one Dynkin quiver of type $D_4$ and two isolated points.
So, $\Lambda$ is representation-finite.

\item Let us show that $T_2(\Lambda)$ is $\tau$-tilting finite.
Since $T_2(K\overrightarrow{A_3})$ is derived equivalent to the path algebra of Dynkin type $E_6$ \cite{L}, it is seen that $T_2(K\overrightarrow{A_3})$ is silting-discrete.
By Theorem \ref{sd}, we obtain that $T_2(K\overrightarrow{A_3})\otimes_KK[x]/(x^2)$ is silting-discrete; in particular, it is $\tau$-tilting finite.
As there is an algebra epimorphism $T_2(K\overrightarrow{A_3})\otimes_KK[x]/(x^2)\to T_2(\Lambda)$, we deduce that the target $T_2(\Lambda)$ is $\tau$-tilting finite.
\item However, the Auslander algebra $\Gamma$ of $\Lambda$ is not $\tau$-tilting finite.
This is deduced by observing the Auslander--Reiten quiver of $\Lambda$ (it gives a quiver presentation of $\Gamma$):
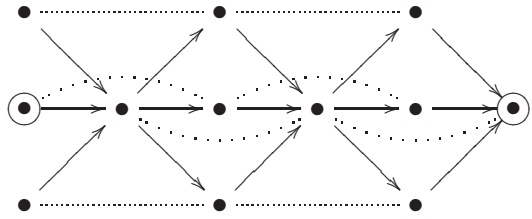
\[\xymatrix{
\bullet \ar[dr]\ar@{.}[rr] & & \bullet \ar[dr]\ar@{.}[rr] & & \bullet \ar[dr] & \\
*+[o][F-]{\bullet} \ar[r]\ar@{.}@/^1pc/[rr] & \bullet \ar[ur]\ar[r]\ar[dr]\ar@{.}@/_1pc/[rr] & \bullet \ar[r]\ar@{.}@/^1pc/[rr] & \bullet \ar[ur]\ar[r]\ar[dr]\ar@{.}@/_1pc/[rr] & \bullet \ar[r] & *+[o][F-]{\bullet} \\
\bullet \ar[ur]\ar@{.}[rr] & & \bullet \ar[ur]\ar@{.}[rr] & & \bullet \ar[ur] & \\
}\]
Here, the vertex $\xymatrix{*+[o][F-]{\bullet}}$ coincides.
Factoring by an ideal, we find the factor algebra $\Gamma_1$ of $\Gamma$ presented by the quiver
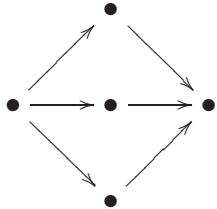
\[\xymatrix{
& \bullet \ar[rd] & \\
\bullet \ar[ru]\ar[rd]\ar[r] & \bullet \ar[r] & \bullet\\
&\bullet \ar[ur]&
}\]
with a zero relation; the sum of the three paths of length 2 is zero.
Truncating $\Gamma_1$ by idempotents, we get the Kronecker algebra, which implies that $\Gamma_1$, and so $\Gamma$, are $\tau$-tilting infinite.
\end{enumerate}
\end{example}

\section{Higher triangulated matrix algebras}
In this section, we focus on the $\tau$-tilting finiteness of ${T_2}^3(\Lambda)$ and $T_n(\Lambda)$ ($n>2$).
Let $\Lambda$ be a ring-indecomposable basic algebra.

In the paper \cite{AR}, it was also discussed that the third triangular matrix algebra ${T_2}^3(\Lambda)$ over an algebra $\Lambda$ is not representation-finite \cite[Theorem 3.4]{AR}.
To see this, we consider the triangular matrix algebra ${T_2}^3(\Lambda/\rad\Lambda)$.
It is because this is a factor algebra of ${T_2}^3(\Lambda)$, since $T_2(\Lambda)/I\simeq T_2(\Lambda/\rad\Lambda)$.
Here, $I$ stands for the ideal $\begin{pmatrix} \rad\Lambda & \rad\Lambda \\ 0 & \rad\Lambda \end{pmatrix}$.
As ${T_2}^3(\Lambda/\rad\Lambda)$ is the direct product of some copies of ${T_2}^3(K)$, the next step is to observe ${T_2}^3(K)$.
We see that ${T_2}^3(K)$ is presented by the quiver
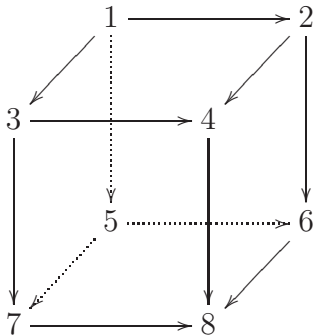
\[\xymatrix{
& 1 \ar[rr]\ar[dl] \ar@{.>}[dd]& & 2 \ar[dl]\ar[dd] \\
3 \ar[rr]\ar[dd] & & 4 \ar[dd] & \\
& 5 \ar@{.>}[rr]\ar@{.>}[dl] & & 6 \ar[dl] \\
7 \ar[rr] & & 8 
}\]
%with all possible commutative relations,
whose separated quiver contains the connected component:
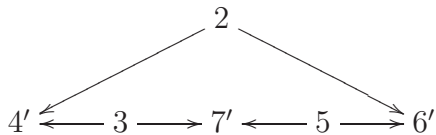
\[\xymatrix{
&&2 \ar[dll]\ar[drr]&&\\
4' & 3 \ar[l]\ar[r] & 7' & 5 \ar[l]\ar[r] & 6' 
}\]
This implies that ${T_2}^3(K)/\rad^2 {T_2}^3(K)$ is representation-infinite.
Consequently, it turns out that ${T_2}^3(\Lambda)$ is of infinite representation type.

We can apply this argument to obtain the following result.

\begin{proposition}\label{third}
\begin{enumerate}
\item The triangular matrix algebra ${T_2}^3(\Lambda)$ is $\tau$-tilting infinite.
\item For nonlocal algebras $\Lambda, \Gamma$ and $\Sigma$,
$\Lambda\otimes_K\Gamma\otimes_K\Sigma$ is $\tau$-tilting infinite.
\end{enumerate}
\end{proposition}
\begin{proof}
(1) Combine the argument above and Adachi's theorem \cite[Theorem 3.1]{Ad1}.

(2) By assumption, there is an algebra epimorphism from $\Lambda\otimes_K\Gamma\otimes_K\Sigma$
to $K\overrightarrow{A_2}\otimes_KK\overrightarrow{A_2}\otimes_KK\overrightarrow{A_2}\simeq {T_2}^3(K)$.
Then apply (1).
\end{proof}

We expect that there is an upper bound of $n$ such that the $n\times n$ triangular matrix algebra over a nonsemisimple algebra is $\tau$-tilting finite; cf. \cite[Theorem 6.1]{LS}.
%, but one has the following observation; .

\begin{proposition}\label{polynomial}
\begin{enumerate}
\item Let $n>0$. If $\Lambda$ is local, then $T_n(\Lambda)$ is silting-discrete.
Hence, it is $\tau$-tilting finite.
\item Assume that $\Lambda$ is nonlocal.
If $T_n(\Lambda)$ is $\tau$-tilting finite, then we have $n\leq4$.
\end{enumerate}
\end{proposition}
\begin{proof}
%We obtain that $T_n(K[x]/(x^2))$ is presented by the quiver
%\[\xymatrix{
%1 \ar@(lu,ru)^{\alpha}\ar[r]_{\beta} & 2 \ar@(lu,ru)^{\alpha}\ar[r]_{\beta} & \cdots \ar[r]_{\beta} & n \ar@(lu,ru)^{\alpha}
%}\]
%with relations $\alpha^2=0$ and $\alpha\beta=\beta\alpha$.
%It follows from \cite[Theorem 4.8]{AK} that this has the same poset structure of support $\tau$-tilting modules as $K\overrightarrow{A_n}$, which is $\tau$-tilting finite.
%The algebra is isomorphic to $K[x]/(x^\ell)\otimes_KK\overrightarrow{A_n}$, and then apply Theorem \ref{sd}.
(1) The algebra $T_n(\Lambda)$ is isomorphic to $\Lambda\otimes_KK\overrightarrow{A_n}$, and then apply Theorem \ref{sd}.

(2) Let $n\geq5$.
As $\Lambda$ is nonlocal,
we obtain algebra epimorphisms $T_n(\Lambda)\to T_5(K\overrightarrow{A_2})\simeq T_2(K\overrightarrow{A_5})$.
The target is not $\tau$-tilting finite by Example \ref{CL},
whence neither is $T_n(\Lambda)$.
%So, we have $n\leq4$.
\end{proof}

In the rest of this section, we explore when $T_n(\Lambda)$ is $\tau$-tilting finite for $n\geq3$.

First, we treat radical-square-zero Nakayama algebras, which play a role in our goal.

\begin{lemma}\label{n=4}
Let $\Lambda$ be a nonlocal Nakayama algebra with radical square zero.
If $\Lambda\neq K\overrightarrow{A_2}$, then $T_4(\Lambda)$ is $\tau$-tilting infinite.
\end{lemma}
\begin{proof}
Assume that $\Lambda$ is linear Nakayama with at least 3 simple modules.
Since $T_4(\Lambda)$ is strongly simply-connected and representation-infinite by \cite[Theorem 6.2]{LS}, we obtain from \cite[Theorem 2.6]{W} that it is $\tau$-tilting infinite.

If $\Lambda$ is cyclic Nakayama with at least 3 simple modules,
then there is an algebra epimorphism $\Lambda\to \Gamma$, which induces $T_4(\Lambda)\to T_4(\Gamma)$.
Here, $\Gamma:=K\overrightarrow{A_3}/\rad^2K\overrightarrow{A_3}$.
As above, this implies that $T_4(\Lambda)$ is $\tau$-tilting infinite.

We show that $T_4(\Lambda)$ is not $\tau$-tilting finite if $\Lambda$ is a radical-square-zero cyclic Nakayama algebra with precisely 2 simple modules.
Then one sees from the Happel--Vossieck List \cite{HV} that it has a tame concealed factor algebra of type $\widetilde{E_7}$ as follows:
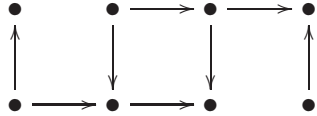
\[\xymatrix{
\bullet & \bullet \ar[r]\ar[d] & \bullet \ar[r]\ar[d] & \bullet  \\
\bullet \ar[r]\ar[u]& \bullet \ar[r] & \bullet  & \bullet \ar[u]
}\]
Hence, it turns out that $T_4(\Lambda)$ is not $\tau$-tilting finite.
\end{proof}

We solve the problem for the case that given algebras have at least 3 simple modules.

\begin{theorem}\label{s3}
Let $\Lambda$ be an algebra given by a quiver $Q$ which has no loops and at least 3 vertices.
Let $n\geq3$.
Then the following are equivalent:
\begin{enumerate}
\item $T_n(\Lambda)$ is $\tau$-tilting finite;
\item It is representation-finite;
\item $n=3$ and $\Lambda$ is a Nakayama algebra with radical square zero.
\end{enumerate}
\end{theorem}
\begin{proof}
It is trivial that (2) implies (1).
It follows from \cite[Theorem 6.1]{LS} that the implications (2)$\Leftrightarrow$(3) hold true.

We show that (1) implies (3). 
One may suppose that $Q$ admits no multiple arrow.
Assume that $Q$ has $\xymatrix{\bullet & \bullet \ar[r]\ar[l] & \bullet}$ or $\xymatrix{\bullet \ar[r] & \bullet  & \bullet \ar[l]}$
as a subquiver.
Then, we see that there is an algebra epimorphism $T_n(\Lambda)\to T_3(A)$, where $A$ is the path algebra of $\xymatrix{\bullet & \bullet \ar[r]\ar[l] & \bullet}$ or $\xymatrix{\bullet \ar[r] & \bullet  & \bullet \ar[l]}$.
By the Happel--Vossieck List \cite{HV}, we observe that a tame concealed algebra of type $\widetilde{E_7}$ appears as a factor algebra of $T_3(A)$, which is $\tau$-tilting infinite, and hence, so is $T_n(\Lambda)$.
Thus, we find out that $\Lambda$ is a Nakayama algebra.
As a similar argument above, we deduce the fact that $Q$ does not admit $\xymatrix{\bullet \ar[r] & \bullet \ar[r] & \bullet}$ without zero relation,
which implies that $\Lambda$ has radical square zero.
Finally, apply Proposition \ref{polynomial}(2) and Lemma \ref{n=4} to get $n=3$.
\end{proof}

Let us turn to the case where a given algebra has precisely 2 simple modules.
We prepare a lemma to reduce the length.

\begin{lemma}\label{2rad2}
Let $\Lambda$ be a cyclic Nakayama algebra with precisely 2 simple modules.
Then we have a poset isomorphism $\sttilt T_n(\Lambda)\simeq \sttilt T_n(\Lambda/\rad^2\Lambda)$.
\end{lemma}
\begin{proof}
By assumption, $\Lambda$ is given by the quiver $\xymatrix{1 \ar@<2pt>[r]^x & 2 \ar@<2pt>[l]^y}$.
Then it is seen that $z:=xy+yx$ belongs to the center and the radical of $\Lambda$,
whence $zI$ is in those of $T_n(\Lambda)$.
Here, $I$ is the identity matrix.
We observe that the factor algebra of $T_n(\Lambda)$ by the ideal generated by $zI$ is isomorphic to $T_n(\Lambda/\rad^2\Lambda)$,
which completes the proof by \cite[Theorem 11]{EJR}.
\end{proof}

Now, we totally realize our goal.

\begin{theorem}\label{s2}
Let $\Lambda$ be an algebra whose quiver has precisely 2 vertices and no loops.
Let $n\geq3$.
Then the following are equivalent:
\begin{enumerate}
\item $T_n(\Lambda)$ is $\tau$-tilting finite;
\item $n=3$ and $\Lambda$ is a Nakayama algebra,
or $n=4$ and $\Lambda=K\overrightarrow{A_2}$.
\end{enumerate}
\end{theorem}
\begin{proof}
If $T_n(\Lambda)$ is $\tau$-tilting finite, then we observe that $n\leq4$ by Proposition \ref{polynomial}(2), and $\Lambda$ is also $\tau$-tilting finite.
This implies that $\Lambda$ has no multiple arrow, so it is a Nakayama algebra.

Let $n=4$ and $\Lambda\neq K\overrightarrow{A_2}$; so $\Lambda$ is cyclic Nakayama.
By Lemma \ref{2rad2}, we can suppose that $\Lambda$ has radical square zero, but we then obtain from Lemma \ref{n=4} that $T_2(\Lambda)$ is not $\tau$-tilting finite, contrary.
Thus, if $n=4$, then we have $\Lambda=K\overrightarrow{A_2}$.

Let us show that the implication (2)$\Rightarrow$(1) holds true.
By Example \ref{CL}, we have only to check the case where $n=3$ and $\Lambda$ is cyclic Nakayama.
From Lemma \ref{2rad2}, one obtains $\sttilt T_3(\Lambda)\simeq \sttilt T_3(\Lambda/\rad^2 \Lambda)$,
which is a finite set because $T_3(\Lambda/\rad^2\Lambda)$ is representation-finite by \cite[Theorem 6.1]{LS}.
Thus, we have done.
\end{proof}

As a corollary of Theorems \ref{s3} and \ref{s2}, we get the following.

\begin{corollary}\label{sc3}
Let $\Lambda$ be a simply-connected algebra and $n\geq3$.
Then $T_n(\Lambda)$ is $\tau$-tilting finite if and only if it is representation-finite.
\end{corollary}

Finally, we give a complete list of positive integers $n$ and $r$ such that
$T_n(\Lambda)$ is silting-discrete for $\Lambda:=K\overrightarrow{A_r}/\rad^2K\overrightarrow{A_r}$.

\begin{theorem}\label{lsd}
Let $\Lambda$ be a radical-square-zero linear Nakayama algebra with $r$ simple modules.
Then $T_n(\Lambda)$ is silting-discrete if and only if one of the following cases occurs:
(i) $n=1$;
(ii) $r=1$;
(iii) $n=2$ and $1<r\leq4$; 
(iv) $1<n\leq4$ and $r=2$.
\end{theorem}
\begin{proof}
It is well-known that $\Lambda$ is derived equivalent to $K\overrightarrow{A_r}$, and so $T_n(\Lambda)$ is derived equivalent to $T_n(K\overrightarrow{A_r})$, which is $\tau$-tilting infinite if $n\geq3$ and $r\geq3$ by Theorem \ref{s3}.
In the case, it is not silting-discrete.

We already know that $T_n(K\overrightarrow{A_2})\simeq T_2(K\overrightarrow{A_n})$ is not silting-discrete for $n\geq5$; see Example \ref{CL}.
For $n=1,2,3$ and $4$,
we have the ADE-chain $A_2, D_4, E_6$ and $E_8$, respectively.
This means that $T_n(K\overrightarrow{A_2})$ is derived equivalent to the path algebra of each type \cite{L},
which is silting-discrete.
This completes the proof.
\end{proof}

%%%%%%%%%%%%%%%%%%%%%%%%%%%%%%%%%%%%%%%%%%%%%%%%%%%%%%%%%%%%%%%%%%%%%%%%%%%%%%%%%%%%%%%%%%%%%%%%%%%%%%%%%%%%%%%%
\section*{Acknowledgements}
The authors express their deep gratitude to Osamu Iyama for insightful comments.
In particular, he pointed out that there is a serious gap in the proof of Theorem \ref{2taustfrf}(1) in the first version.
He also told an easier and natural proof of Theorem \ref{sd}.
%The authors would like to thank the referees for helpful comments.
Thanks are also due to referees for giving the authors helpful comments.

%%%%%%%%%%%%%%%%%%%%%%%%%%%%%%%%%%%%%%%%%%%%%%%%%%%%%%%%

%%%%%%%%%%%%%%%%%%%%%%%%%%%%%%%%%%%%%%%%%%%%%%%%%%%%%%%%

\begin{thebibliography}{AAAAA}
\bibitem[Ad]{Ad1}
{\sc T. Adachi},
Characterizing $\tau$-tilting finite algebras with radical square zero.
{\it Proc. Amer. Math. Soc.} {\bf 144} (2016), no. 11, 4673--4685.

%\bibitem[A2]{Ad}
%{\sc T. Adachi},
%The classification of $\tau$-tilting modules over Nakayama algebras.
%{\it J. Algebra} {\bf 452} (2016), 227--262.

\bibitem[AAC]{AAC}
{\sc T. Adachi, T. Aihara and A. Chan},
Classification of two-term tilting complexes over Brauer graph algebras.
{\it Math. Z.} {\bf 290} (2018), no. 1--2, 1--36.

\bibitem[AK]{AdK}
{\sc T. Adachi and R. Kase},
Examples of tilting-discrete self-injective algebras which are not silting-discrete.
Preprint (2020), arXiv: 2012.14119.

\bibitem[AIR]{AIR}
{\sc T. Adachi, O. Iyama and I. Reiten},
$\tau$-tilting theory.
{\it Compos. Math.} {\bf 150}, no. 3, 415--452 (2014).

\bibitem[Ai1]{Ai}
{\sc T. Aihara},
Tilting-connected symmetric algebras.
{\it Algebr. Represent. Theory} {\bf 16} (2013), no. 3, 873--894.

\bibitem[Ai2]{Ai2}
{\sc T. Aihara},
On silting-discrete triangulated categories.
{\it Proceedings of the 47th Symposium on Ring Theory and Representation Theory}, 7--13, {\it Symp. Ring Theory Represent. Theory Organ. Comm., Okayama}, 2015.

\bibitem[AGI]{AGI}
{\sc T. Aihara, J. Grant and O. Iyama},
Private communication.

%\bibitem[AiK]{AK}
%{\sc T. Aihara and R. Kase},
%Algebras sharing the same support $\tau$-tilting poset with tree quiver algebras.
%{\it Q. J. Math.} {\bf 69} (2018), no. 4, 1303--1325.

\bibitem[AHMW]{AHMW}
{\sc T. Aihara, T. Honma, K. Miyamoto and Q. Wang},
Report on the finiteness of silting objects.
Preprint (2020), arXiv: 2002.08534.

\bibitem[AI]{AI}
{\sc T. Aihara and O. Iyama},
Silting mutation in triangulated categories.
{\it J. Lond. Math. Soc. (2)} {\bf 85} (2012), no. 3, 633--668.

\bibitem[AM]{AM}
{\sc T. Aihara and Y. Mizuno},
Classifying tilting complexes over preprojective algebras of Dynkin type.
{\it Algebra Number Theory} {\bf 11} (2017), no. 6, 1287--1315.

%\bibitem[ASS]{ASS}
%{\sc I. Assem, D. Simson and A. Skowronski},
%Elements of the representation theory of associative algebras. Vol. 1.
%Techniques of representation theory.
%London Mathematical Society Student Texts, {\bf 65}.
%{\it Cambridge University Press, Cambridge}, 2006.

%\bibitem[AS]{AS}
%{\sc I. Assem and A. Skowronski},
%Quadratic forms and iterated tilted algebras.
%{\it J. Algebra} {\bf 128} (1990), no. 1, 55--85.

%\bibitem[AHR]{AHR}
%{\sc I. Assem, D. Happel and O. Roldan},
%Representation-finite trivial extension algebras.
%{\it. J. Pure Appl. Algebra} {\bf 33} (1984), no. 3, 235--242.

\bibitem[AD]{AD}
{\sc J. August and A. Dugas},
Silting and tilting for weakly symmetric algebras.
Preprint (2021), arXiv: 2101.03097.

\bibitem[AR]{AR}
{\sc M. Auslander and I. Reiten},
On the representation type of triangular matrix rings.
{\it J. London Math. Soc. (2)} {\bf 12} (1975/76), no. 3, 371--382

\bibitem[ARS]{ARS}
{\sc M. Auslander, I. Reiten and S. O. Smalo},
Representation theory of Artin algebras.
Cambridge Studies in Advanced Mathematics, {\bf 36}.
{\it Cambridge University Press, Cambridge}, 1995.

%\bibitem[BS1]{BS1}
%{\sc J. Bialkowski and A. Skowronski},
%On tame weakly symmetric algebras having only periodic modules.
%{\it Arch. Math. (Basel)} {\bf 81} (2003), no. 2, 142--154.

%\bibitem[BS2]{BS2}
%{\sc J. Bialkowski and A. Skowronski},
%Socle deformations of self-injective algebras of tubular type.
%{\it J. Math. Soc. Japan} {\bf 56} (2004), no. 3, 687-716.

%\bibitem[BHS]{BHS}
%{\sc J. Bialkowski, T. Holm and A. Skowronski},
%Derived equivalences for tame weakly symmetric algebras having only periodic modules.
%{\it J. Algebra} {\bf 269} (2003), no. 2, 652--668.

\bibitem[BPP]{BPP}
{\sc N. Broomhead, D. Pauksztello and D. Ploog},
Discrete derived categories II: the silting pairs CW complex and the stability manifold.
{\it J. Lond. Math. Soc. (2)} {\bf 93} (2016), no. 2, 273--300.

%\bibitem[CKL]{CKL}
%{\sc A. Chan, S. Koenig and Y. Liu},
%Simple-minded system, configurations and mutations for representation-finite self-injective algebras.
%{\it J. Pure Appl. Algebra} {\bf 219} (2015), no. 6, 1940--1961.

%\bibitem[CH]{CH}
%{\sc F. U. Coelho and D. Happel},
%Quasitilted algebras admit a preprojective component.
%{\it Proc. Amer. Math. Soc.} {\bf 125} (1997), no. 5, 1283--1291.

\bibitem[DIJ]{DIJ}
{\sc L. Demonet, O. Iyama and G. Jasso},
$\tau$-tilting finite algebras, bricks and $g$-vectors.
{\it Int. Math. Res. Not. IMRN} 2019, no. 3, 852--892.

\bibitem[DIRRT]{DIRRT}
{\sc L. Demonet, O. Iyama, N. Reading, I. Reiten and H. Thomas},
Lattice theory of torsion classes.
Preprint (2017), arXiv: 1711.01785.

\bibitem[EJR]{EJR}
{\sc F. Eisele, G. Janssens and T. Raedschelders},
A reduction theorem for $\tau$-rigid modules.
{\it Math. Z.} {\bf 290} (2018), no. 3--4, 1377--1413.

\bibitem[EH]{EH}
{\sc E. G. Escolar and Y. Hiraoka},
Persistence modules on commutative ladders of finite type.
{\it Discrete Comput. Geom.} {\bf 55} (2016), no. 1, 100--157.

%\bibitem[FP]{FP}
%{\sc E. A. Fernandez and M. I. Platzeck},
%Presentations of trivial extensions of finite dimensional algebras and a theorem of Sheila Benner.
%{\it J. Algebra} {\bf 249} (2002), no. 2, 326--344.

%\bibitem[H]{H}
%{\sc D. Happel},
%Tilting sets on cylinders.
%{\it Proc. London Math. Soc. (3)} {\bf 51} (1985), no. 1, 21--55.

\bibitem[HV]{HV}
{\sc D. Happel and D. Vossieck},
Minimal algebras of infinite representation type with preprojective component.
{\it Manuscripta Math.} {\bf 42} (1983), no. 2--3, 221--243.

%\bibitem[IY]{IY}
%{\sc O. Iyama and D. Yang},
%Silting reduction and Calabi-Yau reduction of triangulated categories.
%{\it Trans. Amer. Math. Soc.} {\bf 370} (2018), no. 11, 7861--7898.

\bibitem[IX]{IX}
{\sc O. Iyama and Z. Xiaojin},
Tilting modules over Auslander--Gorenstein algebras.
{\it Pacific J. Math.} {\bf 298} (2019), no. 2, 399--416.

\bibitem[L]{L}
{\sc S. Ladkani},
On derived equivalences of lines, rectangles and triangles.
{\it J. Lond. Math. Soc. (2)} {\bf 87} (2013), no. 1, 157--176.

%\bibitem[LP]{LP}
%{\sc H. Lenzing and J. A. de la Pena},
%Spectral analysis of finite dimensional algebras and singularities.
%{\it Trends in representation theory of algebras and related topics}, 541--588, EMS Ser. Congr. Rep., {\it Eur. Math. Soc., Zurich}, 2008.

\bibitem[LS1]{LS1}
{\sc Z. Leszczynski and A. Skowronski},
Tame triangular matrix algebras.
{\it Colloq. Math.} {\bf 86} (2000), no. 2, 259--303.

\bibitem[LS2]{LS}
{\sc Z. Leszczynski and A. Skowronski},
Tame tensor products of algebras.
{\it Colloq. Math.} {\bf 98} (2003), no. 1, 125--145.

%\bibitem[MXZ]{MXZ}
%{\sc X. Ma, Z. Xie and T. Zhao},
%Support $\tau$-tilting modules and recollements.
%Preprint (2018), arXiv: 1801.02343.

%\bibitem[MS]{MS}
%{\sc P. Malicki and A. Skowronski},
%Cycle-finite algebras with finitely many $\tau$-rigid indecomposable modules.
%{\it Comm. Algebra} {\bf 44} (2016), no. 5, 2048--2057.

%\bibitem[Mi]{M}
%{\sc Y. Mizuno},
%Classifying $\tau$-tilting modules over preprojective algebras of Dynkin type.
%{\it Math. Z.} {\bf 277} (2014), no. 3--4, 665--690.

%\bibitem[Mo]{Mo}
%{\sc K. Mousavand},
%$\tau$-tilting finiteness of biserial algebras.
%Preprint (2019), arXiv: 1904.11514.

%\bibitem[P]{P}
%{\sc P. G. Plamondon},
%$\tau$-tilting finite gentle algebras are representation-finite.
%Preprint (2018), arXiv: 1809.06313.

%\bibitem[R]{R}
%{\sc C. M. Ringel},
%Tame algebras and integral quadratic forms.
%Lecture Notes in Mathematics, 1099.
%{\it Springer-Verlag, Berlin}, 1984.

%\bibitem[SS]{SS}
%{\sc D. Simson and A. Skowronski},
%Elements of the representation theory of associative algebras.
%Vol. 2.
%Tubes and concealed algebras of Euclidean type.
%London Mathematical Society Student Texts, {\bf 71}.
%{\it Cambridge University Press, Cambridge}, 2007.

%\bibitem[S1]{S}
%{\sc A. Skowronski},
%Selfinjective algebras : finite and tame type.
%{\it Trends in representation theory of Algebras and related topics}, 69--238.
%Contemp. Math., {\bf 406}, {\it Amer. Math. Soc.}, 2006.

%\bibitem[S2]{S2}
%{\sc A. Skowronski},
%Minimal representation-infinite Artin algebras.
%{\it Math. Proc. Cambridge Philos. Soc.} {\bf 116} (1994), no. 2, 229--243.

\bibitem[W]{W} 
{\sc Q. Wang},
On $\tau$-tilting finite strongly simply connected algebras.
Preprint (2019), arXiv: 1910.01937v2.

%\bibitem[Z]{Z}
%{\sc S. Zito},
%$\tau$-tilting finite tilted and cluster-tilted algebras.
%Preprint (2019), arXiv: 1902.05866.

\end{thebibliography}
\end{document}